\documentclass{article}

\usepackage{graphicx}      
\usepackage{tikz-cd}
\usepackage{amsmath,amsthm,amsfonts,amssymb}
\usepackage{enumerate}
\usepackage{color}

\newtheorem{corollary}{Corollary}

\newtheorem{theorem}{Theorem}

\theoremstyle{definition}
\newtheorem{definition}{Definition}

\theoremstyle{remark}
\newtheorem{remark}{Remark}

\newtheorem{example}{Example}

\newcommand{\bscdot}{\boldsymbol{\cdot}}
\newcommand\xqed[1]{%
  \leavevmode\unskip\penalty9999 \hbox{}\nobreak\hfill
  \quad\hbox{#1}}
\newcommand\eexa{\xqed{$\triangle$}}

\newcommand{\lie}{\mathcal{L}}	

\newcommand{\BA}{B}	
\newcommand{\CO}{\mathcal{C}}	
\newcommand{\DI}{\tilde{X}}	
\newcommand{\di}{\tilde{x}}	
\newcommand{\IN}{\mathbb{Z}}	
\newcommand{\MA}{X}	
\newcommand{\MAp}{Y}	
\newcommand{\NE}{U}	
\newcommand{\RR}{\mathbb{R}}	
\newcommand{\rd}{\mathrm{d}}	
\newcommand{\SO}{\mathrm{SO}} 
\newcommand{\SP}{\mathbb{S}}	
\newcommand{\SW}{S}	
\newcommand{\SWR}{\tilde{S}}	

\DeclareMathOperator{\aff}{aff}
\DeclareMathOperator{\co}{co}
\DeclareMathOperator{\gr}{gr}
\DeclareMathOperator{\interior}{int}
\DeclareMathOperator{\rank}{rank}
\DeclareMathOperator{\ri}{ri}
\DeclareMathOperator{\sign}{sign}
\DeclareMathOperator{\vex}{vex}

\definecolor{azul}{rgb}{0, 0, 0.7}

\begin{document}

\title{Sliding motions on systems with \\ non-Euclidean state spaces: \\ A differential-geometric perspective}

\author{Fernando Casta\~nos} 

\maketitle

\begin{abstract}
This paper extends sliding-mode control theory to nonlinear systems evolving on smooth manifolds. Building on
differential geometric methods, we reformulate Filippov's notion of solutions, characterize well-defined vector fields
on quotient spaces, and provide a consistent geometric definition of higher-order sliding modes. We generalize the
regular form to non-Euclidean settings and design explicit first- and second-order sliding-mode controllers that
respect the manifold structure. Particular attention is given to the role of topological obstructions, which are
illustrated through examples on the cylinder, M\"obius bundle, SO(3), and 2-sphere. Our results highlight how geometric and
topological properties fundamentally influence sliding dynamics and suggest new directions for robust control on nonlinear spaces.
\end{abstract}

\section{Introduction}

Differential geometry provides the natural mathematical framework for describing and analyzing nonlinear control systems
whose state spaces are smooth manifolds rather than Euclidean spaces~\cite{bullo,arnold2,isidori}. Many control
problems---such as rigid-body stabilization, spacecraft attitude control, and mobile robotics---involve configuration spaces
with rich geometric and topological structures~\cite{sastry,bullo}. Concepts such as vector fields, flows, tangent bundles,
and Lie group actions are indispensable for understanding the behavior of such systems beyond the scope of traditional Euclidean
analysis~\cite{lee}.

In nonlinear control theory, differential geometric methods have profoundly advanced the study of controllability, feedback
equivalence, and stabilization~\cite{jurdjevic1978,isidori,sastry,nijmeijer,brockett2014}. They have also revealed fundamental
limitations, such as the impossibility of global asymptotic stabilization on certain manifolds due to topological
obstructions~\cite{bhat2000,jongeneel}. As engineering applications continue to push the boundaries of performance in
geometrically constrained settings, there is an increasing need to integrate geometric insights into controller design.

Sliding-mode control is a powerful and robust technique for stabilizing nonlinear systems, particularly in the presence of
model uncertainties and external disturbances. Classical formulations of sliding-mode control were developed for systems
evolving on Euclidean spaces, where sliding manifolds are typically linear or affine, and discontinuous feedback can be readily
implemented~\cite{utkin,shtessel}. Extending sliding-mode control to systems evolving on smooth manifolds introduces new
challenges: switching surfaces must be modeled as embedded submanifolds, vector fields must respect the manifold structure,
and topological constraints can fundamentally alter the nature of sliding dynamics.

With few exceptions---most notably the early and pioneering work of Sira-Ram\'irez~\cite{sira1988}---differential geometric methods
have not been systematically developed within the context of sliding-mode control. In particular, the geometric and topological
aspects of switching manifolds, sliding vector fields, and higher-order sliding motions on manifolds remain largely unexplored.

\subsection{Contributions}

In this paper, we extend sliding-mode control theory to systems whose state spaces are smooth manifolds. We characterize
well-defined maps and vector fields on quotient manifolds---a result with potential applications beyond the sliding-mode control community.

We reformulate Filippov's notion of solutions in a geometric setting, provide a consistent definition of higher-order sliding modes,
and point out limitations in some existing formulations~\cite{utkin3,utkin2020}.

We also generalize the concept of regular form to non-Euclidean spaces and construct explicit first- and second-order sliding-mode
controllers that respect the geometry of the underlying manifold. Special attention is given to topological obstructions, illustrated
through examples on the cylinder, M\"obius bundle, $\SO(3)$, and 2-sphere.

Our results underscore the importance of geometric and topological considerations in robust control design and suggest new directions
for sliding-mode control on general nonlinear spaces.

\subsection{Paper Structure}

Section~\ref{sec:global_rep} recalls the necessary background for discussing dynamical systems on smooth manifolds and characterizes 
well-defined maps and vector fields on quotient spaces. Geometric definitions of sliding motions are introduced in Section~\ref{sec:sliding}.
In Section~\ref{sec:design}, we generalize the regular form and provide design examples of second-order sliding-mode controllers on
non-Euclidean state spaces. Finally, concluding remarks are presented in Section~\ref{sec:conclusion}.

\section{Global representations for non-Euclidean \\ state-spaces} \label{sec:global_rep}

In the differential-geometric approach to dynamical systems, it is assumed that the state-space is a smooth
$n$-dimensional manifold $\MA$ and that the system's motions are governed by a vector field taking values on its 
tangent bundle $T\MA$. More precisely, the system is described by differential equations of the form
\begin{equation} \label{eq:system}
 \dot{x} = f(x,t) \;,
\end{equation}
where $x \in \MA$ is the state, $t \in \RR$ is time, and $f(x,t)$ belongs to the tangent space $T_{x} \MA$ of $\MA$ at $x$.
Since manifolds are formally defined through \emph{local} coordinate charts, the geometric 
objects on such manifolds (e.g., maps, vector fields, or differential forms) are usually specified by patching 
local coordinate-based definitions. A usual local model for the system is thus
\begin{equation} \label{eq:systemD}
 \dot{\xi} = \varphi(\xi,t) \;,
\end{equation}
where $\xi \in W \subset \RR^n$ is the state coordinate, that is, $\xi = \psi(x)$ with $\psi : \NE \to W$ a homeomorphism
and $\NE \subset \MA$ a local coordinate domain. Here, $\varphi : W\times\RR \to \RR^n$ is a local coordinate-based definition of the time-dependent
vector field $f$. Unfortunately, using multiple coordinate charts in control applications is impractical, so the designer sometimes settles
for a local controller. There are cases, however, in which a \emph{globally defined} control law is essential, and a
coordinate-free representation of $\MA$ becomes mandatory. Throughout, $x$ will denote
a point of the abstract manifold $\MA$, while $\xi$ will denote its coordinate representation in a local chart.

This section explores two extrinsic ways to construct a coordinate-free representation of
a state space. In order to define and design sliding motions, we will be required to construct subsets, maps, and vector fields on such state spaces.

\subsection{Embedded manifolds}

Allow us to begin by recalling some definitions. If $\MA$ and $\DI$ are smooth manifolds and $\Psi:\MA \to \DI$
is a smooth map, for each $x \in \MA$ we define a map $\Psi_* : T_x \MA \to T_{\Psi(x)}\DI$, called the 
\emph{pushforward} associated with $\Psi$, by $(\Psi_* v)g = v(g\circ \Psi)$, for any tangent vector $v \in T_xX$ and 
any smooth function $g : \DI \to \RR$. A smooth map $\Psi : \MA \to \DI$ is called 
a \emph{submersion} if $\Psi_*$ is surjective at each point (or equivalently, if $\rank \Psi = \dim \DI$). It 
is called an \emph{immersion} if $\Psi_*$ is injective at each point (equivalently, if $\rank \Psi = \dim \MA$).
A \emph{smooth embedding} is a smooth map $\Psi : \MA \to \DI$ that is both an immersion and a topological
embedding. That is, it maps a manifold $\MA$ smoothly into another manifold $\DI$ such that the image
$\Psi(\MA)$ inherits the subspace topology from $\DI$, and the map is a homeomorphism onto its image. An
\emph{embedded submanifold} of a manifold $\DI$ is a subset $\MA \subset \DI$ that forms a manifold under
the subspace topology. It is equipped with a smooth structure so that the natural inclusion map 
$\MA \hookrightarrow \DI$ qualifies as a smooth embedding. See~\cite{lee} for details.

One way to obviate the need for multiple charts is to embed $\MA$ in a higher-dimensional real
space $\RR^p$. While this is always possible due to the Whitney Embedding Theorem~\cite[Thm. 6.15]{lee}, we note that
the proof of the Theorem is not constructive. Fortunately, explicit embeddings are known for a variety of manifolds relevant to applications
(the sphere, the torus, SO(3), real projective space, etc).

\begin{example}[The Group of Rotations $\SO(3)$]
 In the attitude control problem for rigid bodies, the configuration manifold is the group $\SO(3)$ of all
 rotations about the origin in $\RR^3$. An effective global control strategy is to embed $\SO(3)$ 
 in $\RR^{3 \times 3}$,
 \begin{equation} \label{eq:SO3}
  \SO(3) = \left\{ R \in \RR^{3\times 3} \mid R^\top R = I, \; \det(R) = 1 \right\} \;.
 \end{equation}
 Note that the representation~\eqref{eq:SO3} is redundant: Whereas the dimension of $\SO(3)$ is only
 equal to three, nine parameters are required to specify each rotation matrix $R \in \RR^{3\times 3}$.
 However, the redundancy drawback is easily counterbalanced by the fact that control laws based on~\eqref{eq:SO3}
 are, by design, global and free of unwinding phenomena~\cite{chaturvedi2011,gomez2019b}. \eexa
\end{example}

\begin{example}[The Sphere $\SP^2$] 
 The 2-sphere can be naturally embedded in $\RR^3$,
 \begin{displaymath}
  \SP^2 = \left\{ L \in \RR^3 \mid \|L\| = 1 \right\} \;.
 \end{displaymath}
 Spherical state-spaces appear, for example, in the kinematic control problem for reduced
 attitude~\cite{brockett1972,bullo1995} or the control of energy-conserving electrical circuits~\cite{sira1988}. 
 \eexa
\end{example}

Given an embedded submanifold $\MA \subset \RR^p$, an arbitrary manifold $\MAp$, and a $k$-times differentiable map 
$\tilde{s} \in \CO^k(\RR^p,\MAp)$, one can naturally define a new map $s \in \CO^k(\MA,\MAp)$ by restricting $\tilde{s}$
to $\MA$; that is, we set $s = \tilde{s}\big|_\MA$. Similarly, suppose that $\tilde{f} : \RR^p\times\RR \to \RR^p$ is a
time-dependent vector field. Then, a vector field $f$ on $\MA$ can be defined by restriction, provided that
\begin{equation} \label{eq:vf_restriction}
 \tilde{f}(x,t) \in T_x \MA \subset \RR^p \;, \quad \text{for all $(x,t) \in \MA\times\RR$} \;.
\end{equation}
In other words, $\tilde{f}$ must be tangent to $\MA$ for the restriction to yield a well-defined vector field on the submanifold.
To alleviate the notation, we will drop the distinction between $\tilde{f} = f$ when embedding the state space in
$\RR^p$. 

\subsection{Quotients of manifolds by group actions}

While the theory of quotient manifolds is well established~\cite[Ch. 21]{lee}, comparatively less attention has been given
to the practical construction of geometric objects on such spaces. In this section, we build upon classical results to present a
 constructive framework for defining subsets, maps, and vector fields on quotient manifolds. This approach will serve as a
foundation for the development of sliding manifolds, Lyapunov functions, and discontinuous vector fields in subsequent sections.

Let $\DI$ and $\Gamma$ be, respectively, a smooth manifold and a group with identity element $e$. Recall that
a \emph{left action of $\Gamma$ on $\DI$} is a map $\Gamma \times \DI \to \DI$, denoted by
$(\gamma,\di) \mapsto \gamma\bscdot \di$, such that
\begin{displaymath}
\begin{aligned}
 \gamma_1 \bscdot (\gamma_2\bscdot \di) &= (\gamma_1 \gamma_2)\bscdot \di \\
                          e \bscdot \di &= \di
\end{aligned}
\end{displaymath}
for all $\gamma_1,\gamma_2 \in \Gamma$ and $\di \in \DI$. For each $\di \in \DI$, the \emph{orbit of $\di$} is the set
$\Gamma \bscdot \di = \left\{ \gamma \bscdot \di \mid \gamma \in \Gamma \right\}$.
We consider only smooth left actions in this paper, so we refer to them simply as actions.
The collection of orbits is called the \emph{quotient of $\DI$ by $\Gamma$} and is denoted by $\DI/\Gamma$. The map
$q : \DI \to \DI/\Gamma$ that sends $\di$ to its orbit is called the \emph{canonical quotient map}.
Although the quotient need not be a smooth manifold in general, it can be given a smooth structure under reasonable
assumptions. Before stating such conditions, recall that a map $g : \MA \to \MAp$ is said to be \emph{proper} 
if the preimage of each compact subset of $\MAp$ is compact.

\begin{definition}[Free and Proper Actions]
 An action of $\Gamma$ on $\DI$ is said to be \emph{free} if
 \begin{displaymath}
  \gamma \bscdot \di = \di \quad \text{for some $\di \in \DI$} \quad \Longrightarrow \quad \gamma = e \;.
 \end{displaymath}
 The action is said to be \emph{proper} if $(\gamma,\di) \mapsto (\gamma\bscdot\di,\di)$ is a proper map.
\end{definition}

\begin{theorem}[Quotient Manifold~\cite{lee}] \label{thm:quotient}
 Suppose $\Gamma$ is a Lie group acting smoothly, freely, and properly on a smooth manifold $\DI$. Then, the orbit
 space $\MA = \DI/\Gamma$ is a topological manifold of dimension equal to $\dim \DI - \dim \Gamma$, and has a
 unique smooth structure such that $q : \DI \to \MA$ is a smooth submersion.
\end{theorem}

Since $q$ is always surjective, $\DI$ can be used to parameterize $\MA = \DI/\Gamma$ (this is specially
useful when $\DI=\RR^n$). However, because $q$ is in general not injective, the parameterization is not unique, 
so special care must be taken to ensure that the objects that we define on $\MA$ using $\DI$ are indeed well
defined and inherit the required degree of smoothness. 

\begin{theorem}[Passing Smoothly to the Quotient~\cite{lee}] \label{thm:passing}
 Suppose $\DI$ and $\MA$ are smooth manifolds and $q : \DI \to \MA$ is a surjective smooth submersion. If $\MAp$ is a smooth
 manifold and $\tilde{s} : \DI \to \MAp$ is a smooth map that is constant on the fibers of $q$, then there exists
a unique smooth map $s : \MA \to \MAp$ such that $s\circ q = \tilde{s}$.
\end{theorem}



Recall that the \emph{fiber} of a map $q : \DI \to \MA$ is a set of the form $q^{-1}(x) \subset \DI$ for some $x \in \MA$.
For $q$ the canonical quotient map, the fibers of $q$ are then precisely the orbits of $\DI$. This takes us directly to the following.

\begin{corollary}[Smooth Map on Quotient Manifold] \label{cor:descend}
 If $\MAp$ is a smooth manifold and $\tilde{s} : \DI \to \MAp$ is a smooth map that is constant along orbits, that is, if
 \begin{equation} \label{eq:descend}
  \tilde{s}(\di) = \tilde{s}(\gamma\bscdot\di) \quad \text{for all $\di \in \DI$ and $\gamma \in \Gamma$} \;,
 \end{equation}
 then there exists a unique smooth map $s: \MA \to \MAp$ such that
 $\tilde{s} = s\circ q$:
 \begin{displaymath}
 \begin{tikzcd}
  \DI \arrow{d}[left]{q} \arrow{rd}{\tilde{s}} &  \\
                       \MA \arrow{r}[below]{s} & \MAp 
 \end{tikzcd} \;.
 \end{displaymath}
\end{corollary}




We now focus on subsets and level sets. Given a quotient map $q : \DI \to \MA$, a subset
$\SWR \subset \DI$ is said to be \emph{saturated with respect to $q$} if $\SWR = q^{-1}(\SW)$
for some subset $S \subset \MA$. Equivalently, $\SWR$ is saturated if
\begin{equation} \label{eq:saturated}
 \di \in \SWR \quad \Longrightarrow \quad \Gamma\bscdot\di \subset \SWR \;.
\end{equation}


\begin{corollary}[Level Set of Quotient Manifold] \label{cor:level_set}
 Suppose that $\tilde{s} : \DI \to \MAp$ satisfies~\eqref{eq:descend} and let $c$ be any element in 
 $\MAp$. Then, the level set $\SWR = \tilde{s}^{-1}(c)$ is saturated with respect to the
 quotient map. Moreover, $\SWR = q^{-1}(\SW)$ with $\SW \subset \MA$ also a level set. 
\end{corollary}

\begin{proof}
 According to Corollary~\ref{cor:descend}, there exists a map $s$ such that $\tilde{s} = s\circ q$. 
 Thus, $\SWR = q^{-1}(s^{-1}(c))$. The result follows by setting $\SW = s^{-1}(c)$.
\end{proof}

We finally turn our attention to vector fields. Let $q_*$ 
be the pushforward of $q$ and consider a vector field 
$\tilde{f} : \DI\times\RR \to T\DI$. If there exists a vector field $f : \MA\times\RR \to T\MA$ 
such that 
\begin{equation} \label{eq:descend_vf_1}
 q_* \tilde{f}(\di,t) = f(q(\di),t) \quad \text{for all $(\di,t) \in \DI\times\RR$} \;, 
\end{equation}
then we say that $\tilde{f}$ and $f$ are \emph{$q$-related}. Since computing $q_*$ explicitly is not
always easy, we will use $\lie_{\tilde{f}}\tilde{s}$, the Lie derivative of a real-valued function 
$\tilde{s}$ along the flow of $\tilde{f}$, to characterize the vectors fields that can be $q$-related.

\begin{corollary}[Vector Field on Quotient Manifold] \label{cor:descend_vf}
 If $\tilde{f}$ is a smooth time-dependent vector field on $\DI$ such that, for any smooth real-valued function $\tilde{s}$
 satisfying~\eqref{eq:descend}, we have that
 \begin{equation} \label{eq:descend_vf}
  \lie_{\tilde{f}}\tilde{s}(\di) = 
   \lie_{\tilde{f}}\tilde{s}(\gamma\bscdot\di) \quad \text{for all $(\di,t) \in \DI\times\RR$ and $\gamma \in \Gamma$} \;,
 \end{equation}
 then there exists a unique smooth time-dependent vector field $f$ on $\MA$ that is $q$-related to $\tilde{f}$.
\end{corollary}

\begin{proof}
 Let $\tilde{f}$ be a vector field satisfying~\eqref{eq:descend_vf} and let $x$ be any point of
 $\MA$. Note that $q$ is surjective, so $q^{-1}(x) \neq \emptyset$ and we can choose an arbitrary
 $\di \in \DI$ such that $q(\di) = x$. Take any smooth map $s : \MA \to \RR$ and compute
 $\tilde{s} = s\circ q$. By construction, $\tilde{s}$ satisfies~\eqref{eq:descend}, so
 \begin{displaymath}
  \lie_{\tilde{f}}(s\circ q)(\di) = 
   \lie_{\tilde{f}}(s\circ q)(\gamma\bscdot\di) \quad \text{for all $\gamma \in \Gamma$} \;.
 \end{displaymath}
 Direct application of the definition of the pushforward yields
 \begin{equation} \label{eq:descend_vf_2}
  q_*\tilde{f}(\di,t) = q_*\tilde{f}(\gamma\bscdot\di,t) \quad \text{for all $\gamma \in \Gamma$ and all $t \in \RR$} \;.
 \end{equation}
 Now we can use~\eqref{eq:descend_vf_1} to define $f$, that is, we set $f(x,t) = q_* \tilde{f}(\di,t)$.
 Note that, because of~\eqref{eq:descend_vf_2}, $f$ is well-defined, independent of the particular
 choice on $\di$. Smoothness can be established by computing $f$ using local coordinates.
\end{proof}

A smooth map that satisfies~\eqref{eq:descend} or a smooth vector field that satisfies~\eqref{eq:descend_vf} are said
to \emph{descend smoothly to the quotient}. 


\begin{example}[The Cylinder $\SP\times\RR$]
 The cylinder arises naturally as the state space for systems such as pendulums or electric motors.
 To construct the cylinder formally, consider the group of integers $\IN$, viewed as a discrete Lie group under addition,
 acting on $\RR^2$ via
 \begin{equation} \label{eq:action_cylinder}
  z\bscdot(\theta,\omega) = (\theta+2\pi z, \omega) \;, \quad z \in \IN \;, \quad (\theta,\omega) \in \RR^2 \;.
 \end{equation}
 Taking the quotient by this group action can be visualized as rolling the plane into a tube around a circle of unit radius.
 The resulting space, $\RR^2/\IN$, is the cylinder $\SP \times \RR$. 
 \eexa
\end{example}

We say that an equilibrium $x^* \in \MA$ is \emph{asymptotically stable} if it is Lyapunov stable~\cite{khalil} and
there exists a neighborhood $\NE \subset \MA$ of $x^*$ such that, for any initial condition $x(0) \in \NE$, we have
\begin{equation} \label{eq:asympt}
 \lim_{t \to \infty} x(t) = x^* \;.
\end{equation}
The equilibrium is called \emph{globally asymptotically stable} if it is Lyapunov stable and~\eqref{eq:asympt}
holds for all $x(0) \in \MA$. If~\eqref{eq:asympt} holds for all $x(0) \in \MA$, except for a set of measure zero,
then we say that $x^*$ is \emph{almost globally asymptotically stable.}

The circle $\SP$ is perhaps the simplest example of a non-Euclidean space, yet it already exhibits nontrivial
topological obstructions. Due to its non-contractibility, it is impossible to achieve global asymptotic stabilization
at any single point~\cite{bhat2000}. As a consequence, one must either settle for almost global asymptotic stability,
or abandon stability altogether in favor of global attractivity~\cite{jongeneel}. The cylinder of course exhibits the
same topological obstacle.

\begin{example}[The M\"obius Bundle] \label{exa:MobiusI}
 Although it is difficult to imagine a simple physical system whose configuration space is the M\"obius bundle, this state space
 is attractive for several reasons. Despite its low dimensionality---which allows for visualization in three-dimensional
 space---it exhibits several nontrivial topological features. The M\"obius bundle is non-contractible,
 forms a nontrivial vector bundle (see Example~\ref{exa:cyl_mob} below), and is non-orientable. These characteristics pose 
 interesting challenges in the design of sliding-mode controllers. 
 
 To formalize the construction, consider again the group of integers $\IN$, but now acting on $\RR^2$ via the map
 \begin{equation} \label{eq:action_mobius} 
   z\bscdot(\theta,\omega) = (\theta + 2\pi z, (-1)^z \omega) \;, \quad z \in \IN \;, \quad (\theta, \omega) \in \RR^2 \;.
 \end{equation} Taking the quotient $\MA = \RR^2 / \IN$ can be visualized as wrapping the plane around a circle
 of unit radius, but with a half-twist introduced on each turn—producing the familiar M\"obius bundle. The group
 action is smooth and free, and it is proper by~\cite[Prop. 9.12]{lee}, so $\MA$ is a smooth manifold by
 Theorem~\ref{thm:quotient}.
 
 According to Corollary~\ref{cor:descend}, a smooth function $\tilde{s} : \RR^2 \to \RR$ descends to a smooth
 function $s : \MA \to \RR$ if and only if it satisfies the invariance condition
 \begin{equation} \label{eq:mobius_c}
  \tilde{s}(\theta, \omega) = \tilde{s}(\theta + 2\pi z, (-1)^z \omega) \;, 
 \end{equation} for all $(\theta, \omega) \in \RR^2$ and all $z \in \IN$.

 The condition~\eqref{eq:mobius_c} implies the following symmetry for the partial derivatives:
 \begin{displaymath}
  \frac{\partial \tilde{s}}{\partial \theta}(\theta,\omega) = 
   \frac{\partial \tilde{s}}{\partial \theta}(\theta+2\pi z,(-1)^z \omega) 
  \quad \text{and} \quad
  \frac{\partial \tilde{s}}{\partial \omega}(\theta,\omega) = (-1)^z
   \frac{\partial \tilde{s}}{\partial \omega}(\theta+2\pi z,(-1)^z \omega) \;.
 \end{displaymath}
 It then follows from Corollary~\ref{cor:descend_vf} that a vector field $\tilde{f}$ on $\RR^2$
 descends smoothly to a vector field $f$ on the M\"obius bundle if
 \begin{multline*}
  \tilde{f}_1(\theta,\omega)\frac{\partial \tilde{s}}{\partial \theta}(\theta,\omega) + 
   \tilde{f}_2(\theta,\omega)\frac{\partial \tilde{s}}{\partial \omega}(\theta,\omega) = \\
  \tilde{f}_1(\theta+2\pi z,(-1)^z w)\frac{\partial \tilde{s}}{\partial \theta}(\theta,w) + 
   (-1)^z\tilde{f}_2(\theta+2\pi z,(-1)^z w)\frac{\partial \tilde{s}}{\partial \theta}(\theta, w) \;,
 \end{multline*}
 that is, if the symmetry
 \begin{subequations} \label{eq:mobius_vf}
 \begin{align}
  \tilde{f}_1(\theta,w) &= \tilde{f}_1(\theta+2\pi z,(-1)^z w) \label{eq:cvf_theta} \\ 
  \tilde{f}_2(\theta,w) &= (-1)^z \tilde{f}_2(\theta+2\pi z,(-1)^z w) \label{eq:cvf_omega} \;.
 \end{align}
 \end{subequations}
 holds.
 \eexa
\end{example}

\section{Geometric definition of sliding-motions} \label{sec:sliding}

The simplest motivation for Filippov's notion of solution arises from the following.

\begin{example}[The Real Line $\RR$] \label{exa:dim1}
 Consider a one-dimensional system of the form~\eqref{eq:system}, with $\MA = \RR$ and
 \begin{equation} \label{eq:dim1}
  f(x, t) = -\sign(x) + 0.5 \;,
 \end{equation}
 where $\sign(x) = -1$ if $x < 0$, $\sign(x) = 1$ if $x > 0$, and $\sign(0)$ is undefined. For any initial
 condition $x_0 > 0$, the solution behaves as
 \begin{equation} \label{eq:solDim1}
  x(t) =
  \begin{cases}
   -0.5 t + x_0 & \text{if $t < 2x_0$} \\
              0 & \text{if $t \ge 2x_0$} 
  \end{cases} \;.
 \end{equation}
 Note that $\dot{x}(t) = 0$ on the interval $[2x_0, +\infty)$.

 Clearly, the differential equation~\eqref{eq:system} with right-hand side~\eqref{eq:dim1} cannot be satisfied
 in the classical sense unless one arbitrarily sets $\sign(0) = 0.5$ (similar problems arise when $x_0 < 0$). Instead, Filippov's approach is to associate
 a differential inclusion $\dot{x} \in F(x,t)$, which, for this case, becomes
 \begin{equation} \label{eq:FDim1}
  F(t,x) = 
   \begin{cases}
    \left\{ -0.5 \right\} & \text{if $x < 0$} \\
                 [-0.5,1] & \text{if $x=0$} \\
     \left\{ 1.5 \right\} & \text{if $x > 0$}
   \end{cases} 
 \end{equation} 
 (see~\cite{filippov} for details). The function defined in~\eqref{eq:solDim1} is a solution of this differential
 inclusion and, by definition, a valid Filippov solution of the original discontinuous system.
\end{example}

Filippov's theory of differential equations with discontinuous right-hand sides was developed for systems
evolving in Euclidean spaces. In this section, we extend Filippov's definition to systems defined on smooth manifolds. 
The main challenge is to describe the notion of solution entirely in terms of the manifold's global
geometric/topological structure, not in terms of local representations. Also, particular attention must be paid to the definition
of switching sets. In our setting, we require these sets to consist of collections of embedded submanifolds, which ensures that
well-defined reduced-order dynamics can be constructed on them.

Allow us to begin with the following.
\begin{definition}[Piecewise-smooth Vector Field]
 A vector field $f:\MA\times\RR \to T\MA$ is called \emph{piecewise smooth} if there exists 
 a finite (or locally finite) partition of $\MA$ into smooth submanifolds $\{\MA_i\}$ such that:
 \begin{enumerate}
  \item $\MA=\bigcup_i \MA_i$, and the interiors of the submanifolds $\MA_i$ are pairwise disjoint;
  \item the restriction $f\big|_{\MA_i \times \RR}: \MA_i \times \RR \to T\MA$ is smooth in the first argument for each $i$;
  \item discontinuities of $f$ on $\MA$ are confined to the boundaries $\partial \MA_i$.
 \end{enumerate}
\end{definition}

Consider the coordinate-free dynamical model~\eqref{eq:system} with $f$ a time-dependent
piecewise-smooth vector field. An $\MA$-valued function $x(\cdot)$ defined on an interval $[a,b] \subset \RR$ 
is called a \emph{solution} of~\eqref{eq:system} if it is absolutely continuous and if
\begin{equation} \label{eq:F}
 \dot{x} \in F(x,t)
\end{equation}
almost everywhere. Here, $F:\MA \times \RR \rightrightarrows T\MA$ is a multi-valued vector field constructed as follows.
First, we collect the limits of $f$ at $(x,t)$. For fixed $t \in \RR$ and $\NE \subset \MA$ the graph of $f$ at time $t$ restricted away from $\NE$ is
\begin{displaymath}
 \gr_{t,\NE}(f) = \left\{(x,f(x,t)) \in T\MA \mid x \in \MA\setminus\NE \right\}
\end{displaymath}
and its closure in $T\MA$ is denoted $\overline{\gr}_{t,\NE}(f)$. For each $(x,t) \in \MA \times \RR$,
the \emph{limit set} of $f$ at $(x,t)$ is
\begin{displaymath}
 \mathbb{L}(x,t) = \bigcap_{\mu(\NE) = 0} \left\{ v \in T_x \MA \mid (x,v) \in \overline{\gr}_{t,\NE}(f) \right\} \;,
\end{displaymath}
where the intersection runs over all sets $\NE$ of measure zero. We then define
\begin{equation} \label{eq:filippovConv}
  F(x,t) = \co\big(\mathbb{L}(x,t)\big) \subset T_x\MA \;,
\end{equation}
the convex hull in the tangent space $T_x\MA$. Note that Filippov's definition~\cite{filippov} is recovered for $\MA = \RR^n$.
In particular, for~\eqref{eq:dim1} we recover~\eqref{eq:FDim1}.



The existence and uniqueness of solutions to the differential inclusion~\eqref{eq:F} have been thoroughly studied 
for $\MA = \RR^n$~\cite{filippov}. Most of the results in~\cite{filippov} extend naturally to the setting of smooth manifolds.
Here we assume that, for every initial condition $x_0 \in \MA$, there exists at least one solution
passing through $x_0$ at $t=0$. In the absence of uniqueness of solutions, we distinguish between \emph{strong} and \emph{weak}
satisfaction of a given property of~\eqref{eq:F}: a property holds strongly if it is satisfied by all solutions, and weakly if it 
is satisfied by at least one.

Note that, if the function $f$ is continuous in $x$ at a point $(x,t)$, then the set-valued map $F(x,t)$ reduces to the singleton $\{f(x,t)\}$,
and the notion of solution coincides with that of classical differential equations on manifolds. If $f$ is discontinuous on a
set $\SW\times\RR$, then the concatenation of a trajectory reaching $\SW$ from one direction and immediately leaving it by following 
a different direction still qualifies as a classical solution. Thus, genuinely non-classical behavior only occurs when a
trajectory remains within $\SW$ for a time interval of positive measure. This observation motivates the following.

\begin{definition}[Local Weak Invariance] A manifold $\SW \subset \MA$ is \emph{locally weakly invariant} under~\eqref{eq:F} if, for each
 $x(0) \in \SW$, there exists a time interval $I = [t_1, t_2]$ with $t_1 < 0 < t_2$ and a solution $x$ defined on $I$ such $x(t) \in \SW$
 for all $t \in I$.
\end{definition}

\begin{definition}[Sliding Vector Field] \label{def:sliding}
 Consider the Filippov differential inclusion~\eqref{eq:F} and an embedded submanifold $\SW \subset \MA$ with inclusion map
 $\imath : \SW \hookrightarrow \MA$. A vector field $\sigma: \SW\times\RR \to T\SW$ is said to be 
 \emph{a sliding vector field} if it is smooth in $\SW$ and if
 \begin{equation} \label{eq:invariant}
  \imath_* \sigma(y,t) \in F^\circ(\imath(y),t) \quad \text{for all $(y,t) \in \SW\times \RR$}
 \end{equation}
 with $F^\circ(x,t) = \ri\big(F(x,t)\big) = \interior_{\aff(F(x,t))} \big(F(x,t)\big)$,
 the \emph{relative interior} of $F(x,t)$, that is, the interior taken with respect to the affine span $\aff(F(x,t)) \subset T_x\MA$.
 Suppose that the dimension of the affine hull of $F(\imath(y),t)$ is equal
 to $m$ for all $y \in \SW$, then we say that \emph{the sliding vector field is of order} $r$ if the dimension of $\SW$ is equal 
 to $n-m r$.
\end{definition}

The reason for requiring $\SW$ to be an embedded submanifold is that we can consider it as a smooth manifold in its own right, so that
the \emph{sliding dynamics}
\begin{equation} \label{eq:sliding}
 \dot{y} = \sigma(y,t)
\end{equation}
are well-defined. A solution of~\eqref{eq:sliding} is called a \emph{sliding orbit}. Moreover, since the pushforward
$\imath_* : T_y\SW \to T_{\imath(y)}\MA$ is injective, it is possible to identify $T_y\SW$ with a linear subspace of
$T_{\imath(y)}\MA$, so the inclusion~\eqref{eq:invariant} makes sense. Clearly, $\imath$ maps solutions
of~\eqref{eq:sliding} into solutions of the differential inclusion~\eqref{eq:F}. The inclusion~\eqref{eq:invariant}
implies that the relative interior of $F(\imath(y),t)$ is nonempty, which ensures the discontinuity of $f$ at $\SW\times\RR$. Finally, note
that the submanifold $\SW$ is locally weakly invariant by construction, so the classical solutions described above are excluded.

\begin{remark}
 The key elements in Definition~\ref{def:sliding} are the invariance of the manifold $\SW$ and the requirement that solutions along 
 $\SW$ be non-classical (that is, not transversal to $\SW$). Some authors replace the condition of the
 solutions being non-classical with that of finite-time attractivity to $\SW$ (see, e.g.,~\cite{utkin2020,utkin3}). However, we argue
 that---just as equilibria or limit cycles, which are invariant sets, retain their identity regardless of their stability
 properties---a sliding orbit should likewise be recognized as such even if it lacks attractivity.
 In fact, as we will illustrate below, repulsive sliding orbits arise naturally in non-Euclidean spaces as a consequence of
 topological obstructions. Furthermore, due to Brockett's obstruction, repulsive sliding orbits also appear naturally in systems
 with non-holonomic constraints (see, e.g.,~\cite{rocha2022}).
\end{remark}

\section{Sliding-mode control} \label{sec:design}

We now turn our attention to the case in which the piecewise-smooth nature of $f$ results from a 
feedback control and its non-autonomous nature results from an external disturbance. More precisely,
we are interested in coordinate-free models of the form
\begin{equation} \label{eq:system_u}
 f(x,t) = h(x,u(x)+d(t))
\end{equation}
with $u : \MA \to \RR^m$ a feedback control, $d(t) \in \RR^m$ a matched disturbance,
and $h : \MA\times \RR^m \to T\MA$ a controlled vector field. 

A widely used approach for designing first-order sliding-mode controllers for nonlinear systems is based
on the notion of \emph{regular form}, which we briefly recall below. In this section, we extend this methodology 
to systems whose state spaces are not necessarily Euclidean. We also present a few controller designs that achieve second-order 
sliding-mode control on non-Euclidean manifolds.

\subsection{First-order sliding-mode control, global regular forms}

A local model of~\eqref{eq:system_u} is said to be in regular form~\cite{lukyanov1981,utkin} if it can be written as
\begin{subequations} \label{eq:local_reg}
\begin{align}
 \dot{\xi}_1 &= \nu_1(\xi_1,\xi_2) \label{eq:local_reg1} \\
 \dot{\xi}_2 &= \nu_2(\xi_1,\xi_2) + \mu(\xi_1,\xi_2)+d(t) \;, \label{eq:local_reg2}
\end{align}
\end{subequations}
where $\xi_1 \in W_1 \subset \RR^{n-m}$ and $\xi_2 \in W_2 \subset \RR^{m}$ are local state coordinates, 
$\mu : W_1\times W_2 \to \RR^m$ is a local expression for the control law $u$, and the vector fields $\nu_1$ and $\nu_2$
have appropriate domains and ranges. Similar to the backstepping approach~\cite{kokotovic1992,lozano1992}, the interest
on the regular form stems from the fact that $\xi_2$ can be considered as a \emph{virtual control} for the reduced-order
model~\eqref{eq:local_reg1}. More precisely, if we are able to find $\zeta : W_1 \to W_2$ such that the autonomous system
\begin{equation} \label{eq:local_sliding}
 \dot{\xi}_1 = \nu_1(\xi_1,\zeta(\xi_1))
\end{equation}
has the desired qualitative properties, then we can define the switching set as
\begin{equation} \label{eq:local_S}
 \Sigma = \left\{ (\xi_1,\xi_2) \in \NE_1\times\NE_2 \mid \xi_2 - \zeta(\xi_1) = 0 \right\} \;.
\end{equation}
The form~\eqref{eq:local_reg2} makes it particularly simple to design $\mu$ such that $\Sigma$ is invariant and finite-time
attractive. Moreover, the sliding vector field is locally diffeomorphic to~\eqref{eq:local_sliding}.

We wish to apply the previous method in a global setting. The coordinate-free version of the condition in~\eqref{eq:local_S} is
\begin{equation} \label{eq:backstep}
 x_2 - \alpha(x_1) = 0
\end{equation}
with $x_1$ and $x_2$ points in smooth manifolds to be defined below. Note that~\eqref{eq:backstep} needs $x_2$ to
belong to a vector space, since we require an addition operation. Thus, it is reasonable to ask that $\MA$ be a vector bundle, a notion
that we will recall shortly. Given two manifolds $M_1$ and $M_2$, we denote by $p_1 : M_1\times M_2 \to M_1$ the projection on the first factor.

\begin{definition}[Vector Bundle~{\cite[Ch. 5]{lee}}] \label{def:v_bundle}
 Let $\BA$ be a smooth manifold. A \emph{smooth vector bundle of rank $k$ over $\BA$} is
 a smooth manifold $\MA$ together with a surjective smooth \emph{projection} map $p: \MA \to \BA$
 satisfying:
 \begin{enumerate}[i)]
  \item For each $b \in \BA$, the set $\MA_b = p^{-1}(b) \subset \MA$ (called the \emph{fiber} of $\MA$ over $b$)
   is endowed with the structure of a $k$-dimensional real vector space.
  \item For each $b \in \BA$, there exists a neighborhood $\NE$ of $b$ in $\BA$ and a diffeomorphism
   $\Psi: p^{-1}(\NE) \to \NE\times \RR^k$ (called a \emph{local trivialization} of $\MA$ over $\BA$), such that the
   following diagram commutes:
   \begin{displaymath}
   \begin{tikzcd}[column sep=small]
    p^{-1}(\NE) \arrow{rd}[left, inner sep=1.5 ex]{p} \arrow{rr}{\Psi} &     & \NE\times\RR^k \arrow{ld}{p_1} \\
                                                                       & \NE & 
   \end{tikzcd} 
   \end{displaymath}
   and such that for each $b \in \NE$, the restriction of $\Psi$
   to $\MA_b$ is a linear isomorphism from $\MA_b$ to $\{b\}\times\RR^k \simeq \RR^k$.
 \end{enumerate}
\end{definition}

A vector bundle is usually denoted by $p :  \MA \to \BA$, or sometimes just as $\MA$ if we wish to stress the projection or
the total space.

\begin{example}[Tangent Bundles]
 An archetypal example of a vector bundle is the tangent bundle $T\BA$ of a general smooth manifold $\BA$. The tangent bundle is
 certainly a vector bundle, as the natural projection $\tau : T\BA \to \BA$ always satisfies the conditions of the projection $p$
 described in Definition~\ref{def:v_bundle}. In the Lagrangian formulation of classical mechanics, the state space is $\MA = T\BA$ 
 with $\BA$ the system's configuration manifold. \eexa
\end{example}

A local trivialization defined over all of $\BA$ is called a \emph{global trivialization}. When a global trivialization exists,
$\MA$ is said to be a \emph{trivial bundle}.

\begin{example}[Difference Between the Cylinder and the M\"obius Bundle] \label{exa:cyl_mob}
 The cylinder and the M\"obius bundle are examples of vector bundles of rank 1. Indeed, with the actions~\eqref{eq:action_cylinder}
 and~\eqref{eq:action_mobius}, the projection $\tilde{p} : (\theta,\omega) \mapsto \theta$ descends to a smooth
 projection $p : \MA \to \SP$ that satisfies the conditions of Definition~\ref{def:v_bundle}.
 Note that the cylinder $\SP\times\RR$ is a trivial bundle, but the M\"obius bundle is not. Intuitively, the M\"obius bundle `looks' 
 like $\SP\times\RR$, but only locally. \eexa
\end{example}

Condition~\eqref{eq:backstep} is globally defined if $\MA$ is a trivial vector bundle of rank $m$ over a manifold $\BA$
of dimension $n-m$ and $\alpha : \BA \to \RR^m$. Thus, a model is in \emph{global regular form} if it is written as
\begin{subequations} \label{eq:regular}
\begin{align}
 \dot{x}_1 &= g_1(x_1,x_2) \label{eq:reg1} \\
 \dot{x}_2 &= g_2(x_1,x_2) + u(x_1,x_2)+d(t) \;, \label{eq:reg2}
\end{align}
\end{subequations}
where $x_1 \in \BA$, $x_2 \in \RR^{m}$ and $g_1 :\BA\times\RR^m \to T\BA$, $g_2 :\BA\times\RR^m \to \RR^m$. Now the switching set
can be defined globally as
\begin{equation} \label{eq:switch}
 \SW = \left\{ (x_1,x_2) \in \BA\times\RR^m \mid s(x_1,x_2) = 0\right\} \;, \quad s(x_1,x_2) = x_2 - \alpha(x_1) \;.
\end{equation}
Note that, if $\alpha$ is smooth, then $\SW$ is an embedded submanifold of $\MA$ as required in Definition~\ref{def:sliding}.
he function $s$ in~\eqref{eq:switch} is called the \emph{sliding variable}.

The tangent bundle of a smooth manifold $\BA$ is trivial if and only if it is \emph{parallelizable}. A necessary condition for
a manifold to be parallelizable is that $\BA$ be orientable. Notably, every Lie group $G$ is parallelizable; therefore, any fully actuated
mechanical system whose configuration space is a Lie group can be expressed in the form~\eqref{eq:regular} with $\BA = G$. More
specifically, we have $TG \simeq G \times \mathfrak{g}$, where $\mathfrak{g} = T_e G \simeq \RR^n$ is the Lie algebra of $G$ and $e \in G$
the identity element~\cite{lee}.

\begin{example}[The Tangent Bundle $T\SO(3)$] \label{exa:TSO3}
 Since $\SO(3)$ is a Lie group, its tangent bundle $T\SO(3)$ is trivial. The rotation matrix $R \in \SO(3)$ of a rigid body with 
 inertia matrix $J \in \RR^{3\times 3}$, $J = J^\top > 0$, evolves according to the dynamics~\cite{murray}
 \begin{displaymath}
 \begin{aligned}
        \dot{R} &= R \omega^\times \\
  J\dot{\omega} &= (J\omega) \times \omega + u(R,\omega) + d(t) 
 \end{aligned} \;,
 \end{displaymath}
 where $\omega \in \RR^3$ is the angular velocity and $\omega^\times$ is the skew-symmetric matrix associated with $\omega$, given by
 \begin{displaymath}
  \omega^\times =
   \begin{pmatrix}
            0 & -\omega_3 & \omega_2 \\
     \omega_3 & 0         & -\omega_1 \\
    -\omega_2 & \omega_1  & 0
   \end{pmatrix} \;.
 \end{displaymath}
 The vectors $u(R,\omega)$ and $d(t)$ are the control and disturbance torques, respectively. Modulo the inertia matrix, the system is in
 regular form~\eqref{eq:regular} with $x_1 = R$ and $x_2 = \omega$, that is, with $B = \SO(3)$ and $m=3$.

 Consider the switching manifold $\SW = \left\{ (R,\omega) \in \SO(3)\times\RR^3 \mid s(R,\omega) = 0 \right\}$,
 where the sliding variable is defined as $s(R,\omega) = \omega - \alpha(R)$ with~\cite{gomez2019b}
 \begin{displaymath}
  \alpha(R) = -\frac{1}{2}\vex(R_\rd^\top R - R^\top R_\rd ) \;,
 \end{displaymath}
 $\vex$ denoting the inverse map of $(\cdot)^\times$, and $R_\rd \in \SO(3)$ a desired rotation matrix. The resulting sliding vector field
 renders $R_\rd$ almost globally asymptotically stable\footnote{It was claimed in~\cite{gomez2019}
 that this switching manifold is a Lie subgroup of $T\SO(3)$. Unfortunately, this is incorrect due to the fact that the 
 product defined for $T\SO(3)$ is not associative.}.

Depending on context, we will use $\|\cdot\|$ when either referring to the Euclidean norm (when the argument is a vector)
 or the induced Euclidean norm (when the argument is a matrix).
 Let $\bar{d}$ be a uniform bound on the disturbance. That is, $\|d(t)\| \le \bar{d}$ for all $t \in \RR$. Consider the control law
 \begin{displaymath}
  u(R,\omega) = -K(\omega)\frac{s(R,\omega)}{\| s(R,\omega) \|} \;, \quad K(\omega) > \|J\|\|\omega\|^2 + \bar{d}
 \end{displaymath}
 renders $\SW$ weakly invariant and finite-time attractive.  See~\cite{gomez2019b} for details. \eexa
\end{example}

Similarly, the configuration space for an $m$-joint robotic manipulator is the $m$-dimensional torus $\mathbb{T}^m$. As a 
Lie group, $\mathbb{T}^m$ has a trivial tangent bundle and its equations of motion can be globally expressed in regular form.
Consequently, any point $(\Theta_\rd,0) \in T\mathbb{T}^m$ can be almost globally asymptotically stabilized.

In cases where the configuration manifold is not parallelizable, one may embed its tangent spaces into a higher-dimensional
real vector space. This allows the state space to be realized as a trivial vector bundle---an operation that is always possible
via the Whitney Embedding Theorem (with the aforementioned caveat about the theorem being non-construvtive)---The
system can then be reformulated in regular form on this embedded space.

\begin{example}[The Tangent Bundle $T\SP^2$] \label{exa:TSP2}
 The sphere $\SP^2$ is orientable but not parallelizable, by the Hairy Ball Theorem. However, by embedding its tangent
 spaces in $\RR^3$, we can construct the trivial state space $\MA = \SP^2\times\RR^3$. The dynamics of the reduced attitude 
 for a rigid body can then be modeled as~\cite{chaturvedi2011}
 \begin{subequations} \label{eq:vf_SP2}
 \begin{align}
        \dot{L} &= L \times \omega \label{eq:vf_SP2_1} \\
  J\dot{\omega} &= (J\omega) \times \omega + u(L,\omega) + d(t)  \;,
 \end{align}
 \end{subequations}
 where $L \in \SP^2$ is the reduced attitude, $\omega \in \RR^3$ is the angular velocity, and the inertia matrix $J$ is symmetric and
 positive definite, as in Example~\ref{exa:TSO3}. The term $d(t)$ also denotes a disturbance torque. Note that the properties of the vector cross product 
 ensure that, at each $L \in \SP^2$, $\dot{L} \in T_L\SP^2 \subset \RR^3$ (i.e., that~\eqref{eq:vf_restriction} holds). Thus, modulo
 the inertia matrix, the system is in the global regular form~\eqref{eq:regular} with $x_1 = L$ and $x_2 = \omega$, that is,
 with $B = \SP^2$ and $m=3$.

 Given a desired reduced attitude $L_\rd \in \SP^2$, the switching manifold
 $S = \left\{ (L,\omega) \in \SP^2\times\RR^3 \mid s(L,\omega) = 0 \right\}$ with $s(L,\omega) = \omega - \alpha(L)$ and
 \begin{equation} \label{eq:s_SP2}
  \alpha(L) = -L\times L_\rd \;,
 \end{equation}
 produces a sliding vector field such that $L_\rd$ is almost globally asymptotically stable. The rational is the following.
 By the properties of the vector triple product, the sliding dynamics $\dot{L} = -L\times(L\times L_\rd)$ can be rewritten 
 as $\dot{L} = (I_3 - LL^\top)L_\rd$  with $I_3$ the $3\times 3$ identity matrix. Note that $I_3 - LL^\top$ is the orthogonal
 projector onto the tangent space to $\SP^2$ at $L$. Similar to the kinematic equations proposed in~\cite{bullo1995}, the resulting
 vector $(I_3 - LL^\top)L_\rd$ is tangent to the great circle joining $L$ and $L_\rd$. The resulting sliding orbits are thus
 great circles passing through $L_\rd$. Notice that the norm of $\dot{L}$ is $\sin(\theta)$ with $\theta$ the angle
 between $L$ and $L_\rd$, i.e., $\theta = \arccos (L^\top L_\rd) \in [0,\pi]$. According to~\cite[Lem. 1]{bullo1995}, we have
 $\dot{\theta} = -\sin(\theta)$, from which we conclude the almost global asymptotic convergence to $\theta=0$ and thus of
 $L$ to $L_\rd$.

 The control law
 \begin{equation} \label{eq:u_TSP2}
  u(L,\omega) = -(J\omega) \times \omega - K(L,\omega)\frac{s(L,\omega) }{\| s(L,\omega) \|}
 \end{equation}
 with $K(L,\omega) > \|J\lie_f \alpha(L,\omega)\| + \bar{d}$, ensures the weak invariance and finite-time attractiveness of $\SW$.
 To see this, consider the Lyapunov-like function $V(s) = s^\top J s$. Its Lie derivative satisfies
 \begin{multline*}
  \lie_f V(L,\omega) = s(L,\omega)^\top \left(-J\lie_f \alpha(L,\omega) -
   K(L,\omega)\frac{s(L,\omega) }{\| s(L,\omega) \|} + d(t) \right) < \\
    -\| s(L,\omega) \|\left(K(L,\omega) - \|J\lie_f \alpha(L,\omega)\| -\bar{d}\right) \;.
 \end{multline*}
 It only remains to compute $\lie_f \alpha(L,\omega)$. Direct substitution of~\eqref{eq:vf_SP2_1} into~\eqref{eq:s_SP2} gives
 $\lie_f \alpha(L,\omega) = -(L\times \omega)\times L_\rd$. See~\cite{gomez2019} for more details.
 \eexa
\end{example}

For non-trivial bundles which happen to be quotient manifolds $\MA = \DI/\Gamma$, it is possible to carry out the design of the
sliding-mode controller directly in $\DI$. However, care must be taken to satisfy the conditions of
Corollaries~\ref{cor:descend},~\ref{cor:level_set}, and~\ref{cor:descend_vf}.
 
\begin{example}[The M\"obius Bundle] \label{exa:MobiusII}
 Consider the system
 \begin{equation} \label{eq:mobius_phi}
 \begin{aligned}
  \dot{\theta} &= \omega \cos \frac{\theta}{2} \\
  \dot{\omega} &= \tilde{u}(\theta,\omega) \;.
 \end{aligned}
 \end{equation}
 The system is in regular form with $\tilde{g}_1(\theta,\omega) = \omega \cos \frac{\theta}{2}$ and
 $\tilde{g}_2(\theta,\omega) = 0$. For clarity, we neglect perturbations at the moment, but these can be easily
 incorporated. We wish to define a dynamical system on the M\"obius bundle,
 which is possible if the closed-loop vector field $\tilde{f}$ satisfies~\eqref{eq:mobius_vf}. It already
 satisfies~\eqref{eq:cvf_theta}: $\tilde{g}_1(\theta,\omega) = \tilde{g}_1(\theta+2\pi z,(-1)^z\omega)$,
 so we only require the symmetry
 \begin{equation} \label{eq:mobius_v_sim}
  \tilde{u}(\theta,\omega) = (-1)^z\tilde{u}(\theta+2\pi z,(-1)^z\omega) \;.
 \end{equation}
 Given a desired reference $\theta^* \not\in \left\{ \pi+2\pi z \mid z \in \IN  \right\}$, 
 our objective is to asymptotically stabilize the orbit $x^* = \IN\bscdot(\theta^*,0)$. 
 Mark that the array of vertical lines
 \begin{equation} \label{eq:vertical}
  \tilde{S}_1 = \left\{ (\theta,\omega) \in \RR^2 \mid \cos \frac{\theta}{2} = 0 \right\}
 \end{equation}
 is invariant with respect to~\eqref{eq:mobius_phi}, regardless of the specific form of $\tilde{u}$.
 This precludes global asymptotic stabilizability of $x^*$, so we aim at almost global asymptotic
 stability.

 Having Corollaries~\ref{cor:descend} and~\ref{cor:level_set} in mind, we propose the switching set
 $\SWR = \tilde{s}^{-1}(0)$ with
 \begin{equation} \label{eq:mobius_sigma}
  \tilde{s}(\theta,\omega) = \cos\frac{\theta}{2}\cdot
   \left( \omega + \sin\left(\frac{\theta - \theta^*}{2}\right) \right) \;.
 \end{equation}
 The Lie derivative of $\tilde{s}$ along the vector field~\eqref{eq:mobius_phi} is
 \begin{equation} \label{eq:mobius_lie}
  \lie_{\tilde{f}}\tilde{s}(\theta,\omega) = 
   \left( \frac{\omega}{2}\cos\left(\theta - \frac{\theta^*}{2}\right) - \frac{\omega^2}{2}\sin \frac{\theta}{2} +
    \tilde{u}(\theta,\omega)\right)\cdot \cos \frac{\theta}{2} \;.
 \end{equation}
 Its expression suggests the control law
 \begin{equation} \label{eq:mobius_v}
  \tilde{u}(\theta,\omega) = \frac{\omega^2}{2}\sin \frac{\theta}{2} - \frac{\omega}{2}\cos\left(\theta - \frac{\theta^*}{2}\right)
   - \sign\left( \cos \frac{\theta}{2} \right) \cdot \sign \tilde{s}(\theta,\omega)
 \end{equation}
 which, when substituted back in~\eqref{eq:mobius_lie}, gives
 \begin{equation} \label{eq:mobius_lie2}
  \lie_{\tilde{f}}\tilde{s}(\theta,\omega) = 
   -\left|\cos \frac{\theta}{2} \right| \cdot \sign \tilde{s}(\theta,\omega) \;.
 \end{equation}
 The Lie derivative of the Lyapunov-like function $V(\tilde{s}) = \frac{1}{2}\tilde{s}^2$ is
 \begin{displaymath}
  \lie_{\tilde{f}} V = -\left|\cos \frac{\theta}{2} \right| \cdot |\tilde{s}(\theta,\omega)| \;,
 \end{displaymath}
 which shows that $\SWR$ is globally finite-time attractive.

 Up to this point, the analysis has been carried out in $\RR^2$. We shall verify that the objects defined so far
 are correctly mapped into the M\"obius bundle. First, note that $\tilde{s}$ satisfies~\eqref{eq:mobius_c}, so
 it descends to a smooth function $s : \MA \to \RR$. Note also that the topology of the bundle prevents us from 
 specifying $\tilde{s}$ in the simple form $\omega - \alpha(\theta)$, as it would not satisfy the conditions of 
 Corollary~\ref{cor:descend}. By Corollary~\ref{cor:level_set}, we have $\SWR = q^{-1}(\SW)$ with $\SW = s^{-1}(0)$. Unfortunately, 
 $\SWR \subset \RR^2$ is not an embedded submanifold of $\RR^2$. However, it is the union of two embedded submanifolds: $\SWR = \SWR_1 \cup \SWR_2$ 
 with $\SWR_1$ as in~\eqref{eq:vertical} and
 \begin{displaymath}
  \SWR_2 = \left\{ (\theta,\omega) \in \RR^2 \mid \omega = -\sin\left( \frac{\theta-\theta^*}{2} \right) \right\} \;.
 \end{displaymath}
 Although neither $\SWR_1$ nor $\SWR_2$ are level sets of functions that descend continuously to $\MA$, these sets
 satisfy~\eqref{eq:saturated}, so are saturated with respect to $q$. This allows us to write $\SW = \SW_1 \cup \SW_2$ with 
 $\SW_1 = \SWR_1/\IN$ and $\SW_2 = \SWR_2/\IN$.

 The feedback~\eqref{eq:mobius_v} satisfies~\eqref{eq:mobius_v_sim}, so $\tilde{f}$ descends to a 
 piecewise-smooth vector field $f : \MA \to T\MA$ with discontinuities on $\SW_2$
 (note that $\tilde{f}$ is smooth on $\SWR_1\setminus \SWR_2$). As expected, the Lie derivative~\eqref{eq:mobius_lie2} also
 descends smoothly to the quotient. Bearing this in mind, we use~\eqref{eq:mobius_lie2} to compute
 \begin{displaymath}
  \co \bigcap_{\mu(\NE) = 0} \left\{ v \in \RR \mid (x,v) \in \overline{\gr}_{t,\NE}(\lie_f s) \right\} =
   \left[ -\left|\cos\left( \tfrac{\theta}{2} \right)\right|, \left|\cos\left( \tfrac{\theta}{2} \right)\right| \right]
 \end{displaymath}
 for $x \in \SW$ and with $(\theta,\omega)$ any representative point of the class $x$. Since
 \begin{displaymath}
  0 \in \left( -\left|\cos\left( \tfrac{\theta}{2} \right)\right|, 
  \left|\cos\left( \tfrac{\theta}{2} \right)\right| \right)
 \end{displaymath}
 for $x \in \SW_2\setminus\SW_1$, we have $F^\circ(\imath(y)) \cap \imath_* T_y \SW_2 \neq \emptyset$, so $\SW_2$ is weakly invariant
 and there exists a sliding vector field $\sigma : \SW_2 \to T\SW_2$.

 \begin{figure}
 \centering
 \includegraphics[width=\columnwidth]{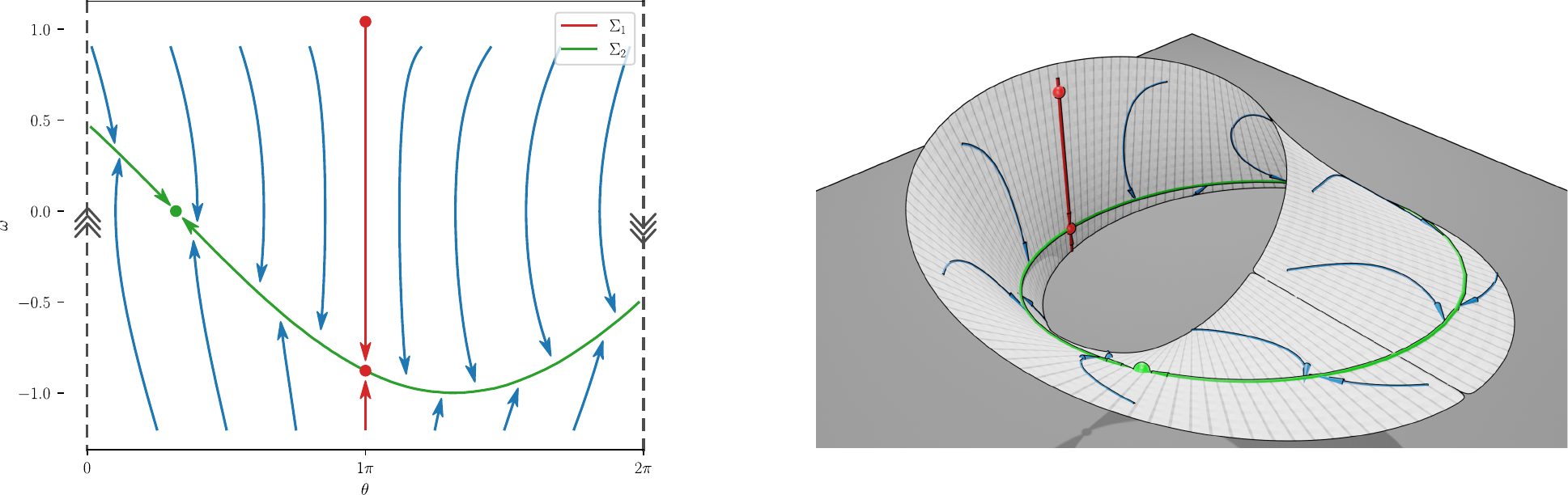}
  \caption{Phase plane of~\eqref{eq:mobius_phi} with the feedback~\eqref{eq:mobius_v}. The desired equilibrium
   $(\theta^*,0)$ is marked with a green dot. The orbits (blue) and the invariant sets $\SWR_1$ 
   (red) and $\SWR_2$ (green) are mapped to $\RR^3$ by~\eqref{eq:mobius_embed}. Trajectories starting in a neighborhood
   of $\SWR_2$ but outside of $\SWR_1$ converge to $\SWR_2$ in finite time and then slide along $\SWR_2$ towards the
   desired reference.}
 \label{fig:phase_plane_mob}
 \end{figure}

 While specifying the sliding dynamics, we will exploit the fact that the dimension of $\SW_2$ is smaller
 than the dimension of $\MA$. We will show that $\SW_2$ is diffeomorphic to $\SP$ and then define the sliding dynamics on
 $\SP$, where they take a simpler form. First, note that the map $\tilde{\psi}:\SWR_2 \to \RR$, defined by 
 $\tilde{\psi}(\theta,\omega) = \theta$, establishes a diffeomorphism between $\SWR_2$ and $\RR$.
 Thus, the flow of a vector field $\tilde{\sigma}$, tangent to $\SWR_2$, is diffeomorphic
 to the flow of the vector field $\tilde{\psi}_*\tilde{\sigma}$ on $\RR$, which takes the form
 \begin{equation} \label{eq:sliding_R}
  (\tilde{\psi}_* \tilde{\sigma})(\theta) = -\cos\frac{\theta}{2}\sin\left(\frac{\theta - \theta^*}{2}\right) = 
   -\frac{1}{2}\sin\left( \theta - \frac{\theta^*}{2} \right) + \frac{1}{2}\sin \frac{\theta^*}{2} \;.
 \end{equation}
 Now, consider the action of $\IN$ on $\RR$ given by
 \begin{equation} \label{eq:action_circ}
  z\bscdot \theta = \theta + 2\pi z
 \end{equation}
 and take the quotient $\RR/\IN = \SP$. Since $\tilde{\psi}$ is a quotient homomorphism, 
 that is, $\tilde{\psi}(z\bscdot (\theta,\omega)) = z \bscdot \tilde{\psi}(\theta,\omega)$\footnote{Note
 that the action on the left-hand side is given by~\eqref{eq:action_mobius}, while the one on the
 right-hand side is given by~\eqref{eq:action_circ}.}, it descends to a smooth
 diffeomorphism $\psi : \SW_2 \to \SP$. The vector field~\eqref{eq:sliding_R} satisfies the
 conditions of Corollary~\ref{cor:descend_vf} with $\DI = \RR$ and the action given 
 by~\eqref{eq:action_circ}, which implies that $\tilde{\psi}_*\tilde{\sigma}$ descends to a smooth
 vector field $\psi_* \sigma : \SP \to T\SP$. The latter characterizes the sliding dynamics.
 The vector field $\psi_* \sigma$ has two equilibria: $\IN\bscdot\pi$, which is unstable,
 and $\IN \bscdot \theta^*$, which is almost globally asymptotically stable.
 
 To aid visualization, it is useful to embed the M\"obius bundle as a surface in $\RR^3$. Consider the map
 $\tilde{k}: \RR^2 \to \RR^3$ defined by
 \begin{equation} \label{eq:mobius_embed}
 \begin{aligned}
  \tilde{k}_1(\theta,\omega) &= \left( 1 + \frac{\omega}{2}\cos \frac{\theta}{2} \right) \cos \theta \\
  \tilde{k}_2(\theta,\omega) &= \left( 1 + \frac{\omega}{2}\cos \frac{\theta}{2} \right) \sin \theta \\
  \tilde{k}_3(\theta,\omega) &= \frac{\omega}{2}\sin \frac{\theta}{2}
 \end{aligned} \;.
 \end{equation}
 This map is constant on each $\IN$-orbit, so it descends to a well-defined smooth map $k : \MA \to \RR^3$.
 As the reader may verify, the restriction of $\tilde{k}$ to the strip $\RR \times (-1, 1)$ descends to a
 smooth embedding. The phase plane is shown in Figure~\ref{fig:phase_plane_mob} for $\theta^* = 1$. The orbits 
 and invariant sets are mapped to $\RR^3$ by~\eqref{eq:mobius_embed}. There are two saddle equilibria on $\SW_1$
 and an asymptotically stable equilibrium on $\SW_2$, at $\IN\bscdot(\theta^*,0)$. The set $\SW_2$ consists 
 of one of the saddles, an asymptotically stable equilibrium, and two heteroclinic orbits connecting them.
 \eexa
\end{example}

\subsection{Second-order sliding-mode control, terminal and \\ twisting algorithms}

Modern methods for designing higher-order sliding-mode controllers are grounded in the concept of homogeneity
of vector fields~\cite{levant2005}, a property originally formulated for real vector spaces. A geometric definition,
involving the Lie bracket of the control vector field with an Euler vector field, is provided in~\cite{bacciotti,bhat2005}.
Although Euler vector fields are inherently defined on linear spaces, this geometric formulation opens the door
to possible extensions of homogeneity to non-Euclidean manifolds. Developing such generalizations remains
an open direction for future research. In the remainder of this section, we introduce two ad hoc second-order
sliding-mode control algorithms.

\begin{example}[Terminal Algorithm for $T\SP^2$]
 Our point of departure is Example~\ref{exa:TSP2}.  Inspired by the second-order terminal algorithm~\cite{venkataraman1993}, we 
 propose to redefine the virtual control to
 \begin{equation} \label{eq:virtual_TSP2}
  \alpha(L) = -\gamma(\theta)\cdot L\times L_\rd
 \end{equation}
 with
 \begin{displaymath}
   \gamma(\theta) =
   \begin{cases}
    \frac{\sin(\theta^*)}{\sin(\theta)}\sqrt{\frac{\theta}{\theta^*}} & \text{if $\theta < \theta^*$} \\
    1 & \text{if $\theta \ge \theta^*$}
   \end{cases}
 \end{displaymath}
 and $\theta^* \in (0, \pi/2)$ defined implicitly by $\tan(\theta^*)=2\theta^*$ (we have $\theta^* \approx 1.17$). 
 The sliding dynamics are now $\dot{L} = \gamma(\theta)(I_3 - LL^\top)L_\rd$ and the norm of $\dot{L}$ is 
 $\delta(\theta) = \gamma(\theta)\sin(\theta)$. By~\cite[Lem. 1]{bullo1995}, we now have $\lie_f \theta = -\delta(\theta)$,
 from which we conclude the almost global finite-time convergence of $\theta$ to 0 and hence of $L$ to $L_\rd$ 
 (the particular value of $\theta^*$ was chosen so that $\delta$ is continuously differentiable).

 The controller~\eqref{eq:u_TSP2} renders $\SW$ finite-time attractive. The derivative of the virtual control is now
 \begin{displaymath}
  \lie_f {\alpha}(L,\omega) = -\gamma'(\theta)(\lie_f \theta) L\times L_\rd - \gamma(\theta) \dot{L}\times L_\rd \;. 
 \end{displaymath}
 As with Euclidean terminal controllers, there is a singularity at $\theta = 0$, $\omega \neq 0$. However, notice that, for
 $\theta \le \theta^*$ and $(\theta,\omega) \in \SW$,
 \begin{align*}
  \|\lie_f \alpha(L,\omega)\| &\le \frac{\sin(\theta^*)^2}{\theta^*} \left( \frac{\theta}{\tan(\theta)} - \frac{1}{2} \right) +
    \gamma(\theta)^2 \| (L\times(L\times L_\rd))\times L_\rd \| \\
                             &= \frac{\sin(\theta^*)^2}{\theta^*} \left( 2\frac{\theta}{\tan(\theta)} - \frac{1}{2}\right) \;.
 \end{align*}
 That is, the singularity is cancelled at $\SW$ and thus is avoided if the initial condition is close
 enough to the switching manifold. A first-order sliding-mode along $\SW\setminus\{0\}$ followed by a second-order
 sliding-mode on $\{0\}$ are established in finite-time.

 \begin{figure}
 \centering
  \includegraphics[width=\columnwidth]{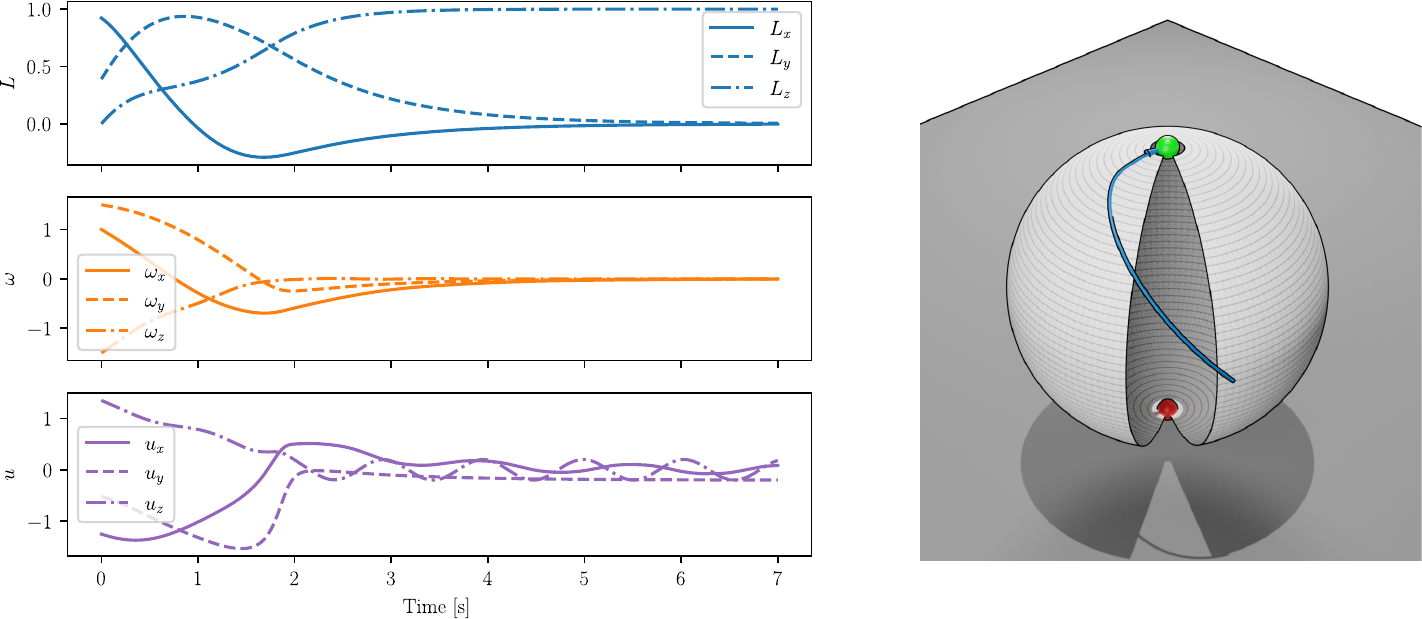}
  \caption{A trajectory of~\eqref{eq:vf_SP2} with the feedback~\eqref{eq:u_TSP2}, the virtual control~\eqref{eq:virtual_TSP2},
   and the parameters~\ref{eq:TSP2_pars}. The discontinuity has been regularized in order to smooth out the control action.
   The reduced orientation $L$ (blue) converges to $L_\rd^\top = (0, 0, 1)$ (green) in finite time.}
   \label{fig:trajectory_sphere}
 \end{figure}

 Suppose that the desired reduced orientation is the `north pole' $L_\rd^\top = (0, 0 ,1)$. For simulation purposes, we set
 \begin{equation} \label{eq:TSP2_pars}
  J = I_3 \;, \quad 
  d(t) = 
   \begin{pmatrix}
    0.1\sin\big( \cos(4t) \big) \\ 0.2 \\ 0.2\sin(6t)
   \end{pmatrix}
   \;, \quad L(0) = 
   \begin{pmatrix}
    1 \\ 0 \\ 0
   \end{pmatrix}
   \;, \quad \text{and} \quad \omega(0) = 
   \begin{pmatrix}
    0 \\ 1 \\ 1
   \end{pmatrix} \;.
 \end{equation}
 The simulation results are shown in Figure~\ref{fig:trajectory_sphere}. It can be seen that the reduced orientation
 converges to $L_\rd$ in finite time.
 \eexa

\end{example}

\begin{example}[Twisitng Algoritm for $\SP\times\RR$] \label{exa:twist}
 Consider the double integrator
 \begin{equation} \label{eq:cyl}
 \begin{aligned}
  \dot{\theta} &= \omega \\
  \dot{\omega} &= \tilde{u}(\theta,\omega) \;,
 \end{aligned}
 \end{equation}
 where $(\theta,\omega) \in \DI = \RR^2$. We aim to define a system on the cylinder $\MA = \DI/\IN$ with the group action specified
 by~\eqref{eq:action_cylinder}. According to Corollary~\ref{cor:descend_vf}, the control input must satisfy the symmetry condition
 \begin{equation} \label{eq:u_cyl}
  \tilde{u}(\theta+2\pi k,\omega) = \tilde{u}(\theta,\omega) \quad \text{for all $(\theta,\omega) \in \DI$ and all $k \in \IN$.}
 \end{equation}

 Inspired by the twisting algorithm~\cite{levant1993,shtessel}, we propose the control law
 \begin{equation} \label{eq:twist}
  \tilde{u}(\theta,\omega) = -K_1\sign\big(\sin(\theta)\big)-K_2\sign(\omega)
 \end{equation}
 with constants satisfying $K_1 > K_2 > 0$. Matched perturbations can be included in the model and compensated, provided that
 additional bounds on $K_1$, $K_2$ are suitably formulated. This control respects the symmetry condition~\eqref{eq:u_cyl}, and
 thus descends to a well-defined control input $u$ on the cylinder. It is smooth except at the classes of the form
 $\IN\bscdot(0,\omega)$, $\IN\bscdot(\pi,\omega)$, and $\IN\bscdot(\theta,0)$ with $\theta, \omega \in \RR$.

 \begin{figure}
 \centering
 \includegraphics[width=\columnwidth]{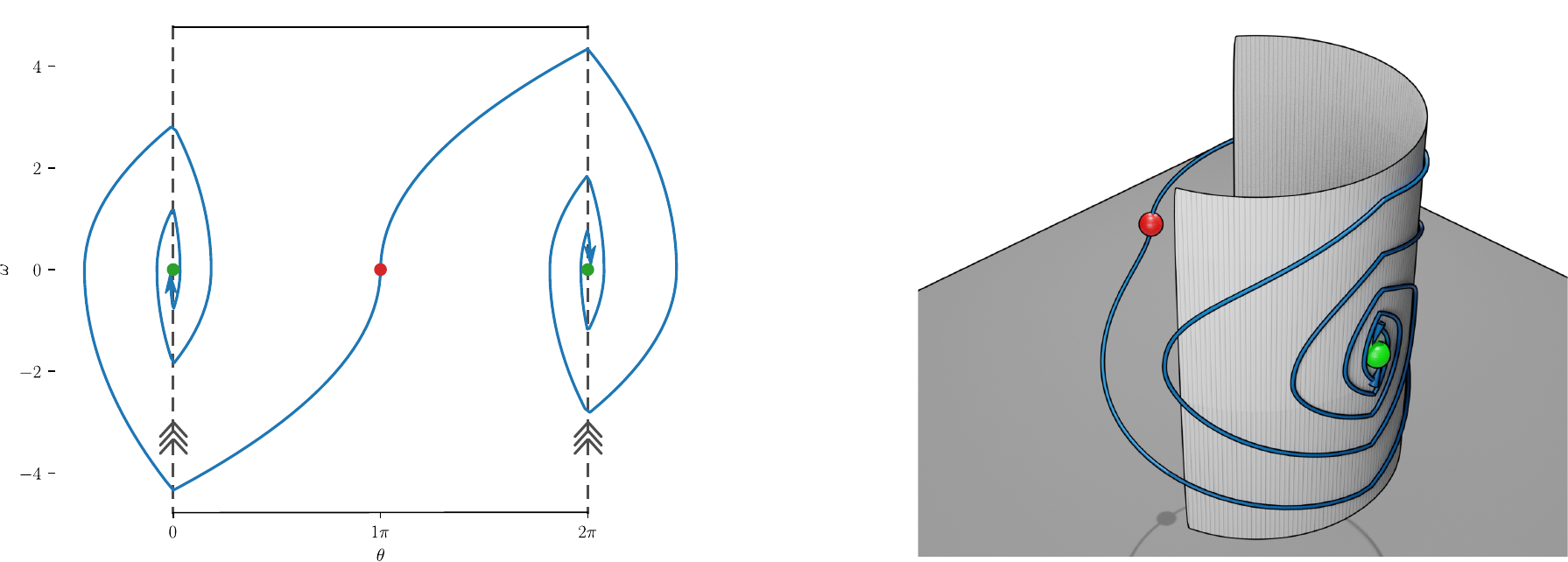}
  \caption{Phase plane of~\eqref{eq:cyl} with the feedback~\eqref{eq:u_cyl}, $K_1 = 5$, $K_2 = 2$. The desired equilibrium, the
   origin, is marked with a green dot and the unstable equilibrium with a red dot. The orbits (blue) and the equilibria are
   mapped to $\RR^3$. The origin is almost globally finite-time stable.}
 \label{fig:phase_plane_cyl}
 \end{figure}

 On the strip $(-\pi,\pi)\times\RR$, the control~\eqref{eq:twist} coincides with its Euclidean counterpart ,
 $-K_1\sign(\theta)-K_2\sign(\omega)$. From~\cite[Thm. 4.2]{shtessel}, we know that this ensures at least local finite-time
 stability of the origin and induces a second-order sliding-mode at the equilibrium class $\IN\bscdot(0,0)$. However,
 in contrast with the Euclidean case, the control law~\eqref{eq:twist} produces a repulsive second-order sliding-motion on the
 antipodal equilibrium $\IN\bscdot(\pi,0)$. This illustrates how the topology of the cylinder obstructs the existence of a 
 globally asymptotically stable equilibrium. A phase portrait is shown in Figure~\ref{fig:phase_plane_cyl}. 
\end{example}

\section{Conclusions} \label{sec:conclusion}

In this work, we developed a geometric framework for extending sliding-mode control to nonlinear systems whose state spaces
are smooth manifolds. By generalizing classical constructions---such as Filippov’s solution concept and the regular form---to
accommodate the intrinsic structure of manifolds, we provided a consistent approach to defining and analyzing sliding dynamics
beyond the Euclidean setting.

We showed how embedded submanifolds can serve as switching surfaces, and how sliding vector fields and control laws can be
constructed to respect the geometry of the configuration space. Our examples on the cylinder, M\"obius bundle, and 2-sphere
highlight the influence of topology on system behavior, including the emergence of repulsive sliding orbits and the fundamental
limitations to global stabilization.

\subsection*{Future work}

An important and challenging problem is trajectory tracking. In the Euclidean case, the first step is to 
characterize admissible reference trajectories $x^* : \RR \to \MA$ for which a nominal 
control $u^* : \RR \to \RR^m$ exists, i.e., trajectories satisfying 
$\dot{x}^* = h(x^*,u^*)$. One then introduces the tracking error $e = x - x^*$ 
(see, e.g.,~\cite{sira}) together with the incremental control $v = u - u^*$. The next step is to compute 
the error dynamics,
\begin{displaymath}
 \dot{e} = h(e + x^*, u^* + v + d(t)) - h(x^*, u^*)
\end{displaymath}
and use $v$ to robustly stabilize the origin $e=0$.
When the state space is a general manifold, each of these steps requires significant adaptation, 
since subtraction and error dynamics are no longer globally defined. Developing an intrinsic 
formulation of trajectory tracking on manifolds is therefore a natural and pressing extension of 
the present work.

More broadly, the results presented here highlight the importance of incorporating geometric and 
topological insights into the design of robust controllers for systems with nonlinear state spaces. 
In particular, we believe that a deeper exploration of intrinsic homogeneity, higher-order sliding 
modes, and feedback equivalence on manifolds will open promising directions for future research 
in geometric control theory.

\bibliographystyle{plain}
\bibliography{local}

\end{document}